\documentclass[draft]{amsart}

\usepackage{pgf,tikz}
\usetikzlibrary{arrows}
\usepackage{cite}

\usepackage{amssymb,graphicx,enumerate,color}
\usepackage{multirow}

\renewcommand{\P}{\mathcal{P}}
\newcommand{\U}{\mathcal{U}}

\newcommand{\R}{{\mathcal{R}}}

\newcommand{\G}{\mathcal{G}}

\newcommand{\beqs}{\begin{equation*}}
\newcommand{\eeqs}{\end{equation*}}
\numberwithin{equation}{section}
 \theoremstyle{plain}
\newtheorem{theorem}{Theorem}[section]

\theoremstyle{remark}

\begin{document}

\makeatletter
\def\imod#1{\allowbreak\mkern10mu({\operator@font mod}\,\,#1)}
\makeatother

\author{Alexander Berkovich}
   \address{Department of Mathematics, University of Florida, 358 Little Hall, Gainesville FL 32611, USA}
   \email{alexb@ufl.edu}

\author{Ali Kemal Uncu}
   \address{Department of Mathematics, University of Florida, 358 Little Hall, Gainesville FL 32611, USA}
   \email{akuncu@ufl.edu}

\title[Variation on a theme of Nathan Fine. New weighted partition identities]{\hspace{1.55cm}Variation on a theme of Nathan Fine.\newline New weighted partition identities}


\dedicatory{Dedicated to our friend, Krishna Alladi, on his 60th birthday.}     

\begin{abstract}  
We utilize false theta function results of Nathan Fine to discover three new partition identities involving weights. These relations connect G\"ollnitz--Gordon type partitions and partitions with distinct odd parts, partitions into distinct parts and ordinary partitions, and partitions with distinct odd parts where the smallest positive integer that is not a part of the partition is odd and ordinary partitions subject to some initial conditions, respectively. Some of our weights involve new partition statistics, one is defined as the number of different odd parts of a partition larger than or equal to a given value and another one is defined as the number of different even parts larger than the first integer that is not a part of the partition.
\end{abstract}   
   


\date{\today}
   
\maketitle
\section{Introduction}

A \textit{partition}, $\pi=(\lambda_1,\lambda_2,\dots)$, is a finite sequence of non-increasing positive integers. Let $\nu(\pi)$ be the number of elements of $\pi$. These elements, $\lambda_i$ for $ i\in\{1,\dots,\nu(\pi)\}$, are called \textit{parts} of the partition $\pi$. The \textit{norm} of a partition $\pi$, denoted $|\pi|$, is defined as the sum of all its parts. We call a partition $\pi$ \textit{a partition of }$n$ if $|\pi|=n$. Conventionally, we define the empty sequence to be the unique partition of zero. Also define $\nu_d(\pi)$ as the number of different parts of $\pi$. For example, $\pi = (10,10,5,5,4,1)$ is a partition of 35 with $\nu(\pi)=6$ and $\nu_d(\pi)=4$.

In 1997, Alladi \cite{AlladiWeighted} began a systematic study of weighted partition identities. Among many interesting results, he proved that
\begin{theorem} [Alladi, 1997] \label{Overpartitions_THM} 
\begin{equation}\label{overpartitions}
\frac{(a(1-b)q;q)_n}{(aq;q)_n} = \sum_{\pi\in \mathcal{U}_n} a^{\nu(\pi)} b^{\nu_d(\pi)} q^{|\pi|},
\end{equation}
where $\mathcal{U}_n$ is the set of partitions with the largest part $\leq n$.
\end{theorem}
In \eqref{overpartitions} and in the rest of the paper we use the standard q-Pochhammer symbol notations defined in \cite{Theory_of_Partitions}, \cite{GasperRahman}. Let $L$ be a non-negative integer, then
\begin{equation*}
(a;q)_L := \prod_{i=0}^{L-1} (1-aq^{i}) \text{  and  } (a;q)_\infty :=\lim_{L\rightarrow\infty} (a;q)_L.
\end{equation*}

Theorem~\ref{Overpartitions_THM} provides a combinatorial interpretation for the left-hand side product of \eqref{overpartitions} as a weighted count of ordinary partitions with a restriction on the largest part. In \cite{Overpartitions_Paper}, Corteel and Lovejoy elegantly interpreted \eqref{overpartitions} with $a=1$ and $b=2$,
\begin{equation}\label{weights_of_overpartitions}
\frac{(-q;q)_n}{(q;q)_n} = \sum_{\pi \in \mathcal{U}_n} 2^{\nu_d(\pi)}q^{|\pi|},
\end{equation}
in terms of \textit{overpartitions}. 

Also in \cite{AlladiWeighted}, Alladi discovered and proved a weighted partition identity relating unrestricted partitions and the Rogers--Ramanujan partitions. Let $\mathcal{U}$ be the set of all partitions, and let $\R\R$ be the set of partitions with difference between parts $\geq 2$. 

\begin{theorem}[Alladi, 1997] \label{Alladi_weighted_sum} 
\begin{equation}\label{omega_12}\sum_{\pi\in \R\R} \omega(\pi) q^{|\pi|} = \sum_{\pi\in \mathcal{U}} q^{|\pi|},\end{equation}where
\begin{equation}\label{myweight}\omega(\pi) := \lambda_{\nu(\pi)}\cdot\prod_{i=1}^{\nu(\pi)-1} (\lambda_{i}-\lambda_{i+1}-1),\end{equation}
and the weight of the empty sequence is considered to be the empty product and is set equal to 1.
\end{theorem} 

In \eqref{omega_12} the set $\R\R$ can be replaced with the set of partitions into distinct parts, denoted $\mathcal{D}$. \begin{equation}\label{RR_to_D}
\sum_{\pi\in\R\R} \omega(\pi) q^{|\pi|} = \sum_{\pi\in \mathcal{D}} \omega(\pi) q^{|\pi|}.
\end{equation} The weight $\omega(\pi)$ of \eqref{myweight} for a partition $\pi$ becomes 0 if the gap between consecutive parts of $\pi$ is ever 1. 
Alladi \cite{Alladi_private_comm} time and time again expressed his desire of seeing a weighted partition identity connecting $\mathcal{D}$ and $\mathcal{U}$ with non-vanishing weights on $\mathcal{D}$ and where we count partitions from $\mathcal{U}$, in the regular fashion, with weight one. In this paper, we discuss a couple of examples of weighted partition identities connecting $\mathcal{D}$ and $\mathcal{U}$, thus taking a step towards the eventual solution to Alladi's problem.
The reader interested in weighted partition identities and their applications may also wish to examine \cite{AlladiWeighted}, \cite{AlladiAnnalsOfComb}, \cite{AlladiBerkovich2}, \cite{AlladiBerkovich}, and \cite{Uncu}. In particular, Theorem~\ref{Alladi_weighted_sum} is extended in \cite{Uncu}. 

In Section~\ref{Section2} we will introduce different representations of partitions which we are going to use later. Section~\ref{GollnitzWeights} discusses weighted partition identities connecting G\"ollnitz-Gordon type partitions and partitions with distinct odd parts. Some combinatorial connections between $\mathcal{D}$ and $\mathcal{U}$ will be presented in Section~\ref{Overpartitions_Section}. In Section~\ref{Section_Restricted}, we consider partitions with distinct odd parts such that the smallest positive integer that is not a part of the partition is odd.

\section{Useful Definitions}\label{Section2}

A Ferrers diagram of a partition $\pi=(\lambda_1,\lambda_2,\dots)$ is a diagram of boxes which has $\lambda_i$ many boxes on its $i$-th row. Whence the number of boxes on $i$-th row gives the size of the part $\lambda_i$. There is a one-to-one correspondence between Ferrers diagrams and partitions. The words \textit{partition} and \textit{Ferrers diagram} can be used interchangeably. An example of a Ferrers diagram is given in Table~\ref{Table_Ferrers_Diagrams}.

We note that Ferrers diagrams represented with boxes as opposed to dots are usually named \textit{Young tableaux}. Young tableaux are often used in a wider context than partitions. We will be using the name Ferrers diagrams to avoid any type of confusion and to remind the reader that we are strictly focusing on partitions.

We define the 2-modular Ferrers diagram similar to the Ferrers diagrams. Let $\lceil x\rceil$ denote the smallest integer $\geq x$. For a given partition $\pi=(\lambda_1,\lambda_2,\dots)$, we draw $\lceil \lambda_i / 2 \rceil$ many boxes at the $i$-th row. We decorate the boxes on the $i$-th row with 2's with the option of having a 1 at the right most box of the row, such that the sum of the numbers in the boxes of the $i$-th row becomes $\lambda_i$. Table~\ref{Table_Ferrers_Diagrams} includes an example of a 2-modular Ferrers diagram.\\

\begin{center}
\begin{table}[htb]\caption{The Ferrers Diagram and the 2-modular Ferrers Diagram of the partition $\pi=(10,9,5,5,4,1)$.}\label{Table_Ferrers_Diagrams}
\definecolor{cqcqcq}{rgb}{0.75,0.75,0.75}
\begin{tikzpicture}[line cap=round,line join=round,>=triangle 45,x=0.5cm,y=0.5cm]
\clip(0.5,0.5) rectangle (19.5,7.5);
\draw [line width=1pt] (1,1)-- (2,1);
\draw [line width=1pt] (2,1)-- (2,7);
\draw [line width=1pt] (1,7)-- (11,7);
\draw [line width=1pt] (11,7)-- (11,6);
\draw [line width=1pt] (11,6)-- (1,6);
\draw [line width=1pt] (1,7)-- (1,1);
\draw [line width=1pt] (1,2)-- (5,2);
\draw [line width=1pt] (5,2)-- (5,7);
\draw [line width=1pt] (10,7)-- (10,5);
\draw [line width=1pt] (10,5)-- (1,5);
\draw [line width=1pt] (9,5)-- (9,7);
\draw [line width=1pt] (8,7)-- (8,5);
\draw [line width=1pt] (7,7)-- (7,5);
\draw [line width=1pt] (6,7)-- (6,3);
\draw [line width=1pt] (6,3)-- (1,3);
\draw [line width=1pt] (1,4)-- (6,4);
\draw [line width=1pt] (4,2)-- (4,7);
\draw [line width=1pt] (3,7)-- (3,2);
\draw [line width=1pt] (14,7)-- (19,7);
\draw [line width=1pt] (19,7)-- (19,5);
\draw [line width=1pt] (19,5)-- (14,5);
\draw [line width=1pt] (14,7)-- (14,1);
\draw [line width=1pt] (15,1)-- (15,7);
\draw [line width=1pt] (17,3)-- (17,7);
\draw [line width=1pt] (18,5)-- (18,7);
\draw [line width=1pt] (16,7)-- (16,2);
\draw [line width=1pt] (19,6)-- (14,6);
\draw [line width=1pt] (17,4)-- (14,4);
\draw [line width=1pt] (17,3)-- (14,3);
\draw [line width=1pt] (16,2)-- (14,2);
\draw [line width=1pt] (15,1)-- (14,1);
\draw (14.5,6.5) node[anchor=center] {2};
\draw (15.5,6.5) node[anchor=center] {2};
\draw (16.5,6.5) node[anchor=center] {2};
\draw (17.5,6.5) node[anchor=center] {2};
\draw (18.5,6.5) node[anchor=center] {2};
\draw (17.5,5.5) node[anchor=center] {2};
\draw (16.5,5.5) node[anchor=center] {2};
\draw (15.5,5.5) node[anchor=center] {2};
\draw (14.5,5.5) node[anchor=center] {2};
\draw (14.5,4.5) node[anchor=center] {2};
\draw (15.5,4.5) node[anchor=center] {2};
\draw (15.5,3.5) node[anchor=center] {2};
\draw (15.5,2.5) node[anchor=center] {2};
\draw (14.5,3.5) node[anchor=center] {2};
\draw (14.5,2.5) node[anchor=center] {2};
\draw (18.5,5.5) node[anchor=center] {1};
\draw (16.5,4.5) node[anchor=center] {1};
\draw (16.5,3.5) node[anchor=center] {1};
\draw (14.5,1.5) node[anchor=center] {1};
\draw (12,4) node[anchor=center] {,};
\end{tikzpicture}
\end{table}\vspace{-1cm}
\end{center}

It should be noted that the \textit{conjugate of a Ferrers diagram} \cite{Theory_of_Partitions} (drawing the Ferrers diagram column-wise and then reading it row-wise) also represents a partition. The analogous conjugation procedure only yields admissible 2-modular Ferrers diagrams if the starting partition has distinct odd parts (it may still have repeating even parts). In the example of Table~\ref{Table_Ferrers_Diagrams}, the conjugate partition of $\pi$ is $(6,5,5,5,4,2,2,2,2,1)$. The conjugate of the 2-modular Ferrers diagram of $\pi$ is not an admissible 2-modular diagram, as there would be a row with two separate boxes decorated with 1's.

Another representation of a partition $\pi$ is the \textit{frequency notation} $\pi = (1^{f_1}, 2^{f_2}, \dots)$, where $f_i(\pi)=f_i$ is the number of occurrences of part $i$ in $\pi$. In this representation the partition in Table~\ref{Table_Ferrers_Diagrams} is $(1^1,2^0,3^0,4^1,5^2,6^0,7^0,8^0,9^1,10^1,11^0,\dots)$.  It is customary to ignore zero frequencies in notation for simplicity. This way the running example can be written as $\pi = (1^1,4^1,5^2,9^1,10^1)$. Occasionally zero frequencies are shown in the representation to stress the absence of a part.

We define the basic $q$-hypergeometric series as they appear in \cite{GasperRahman}. Let $r$ and $s$ be non-negative integers and $a_1,a_2,\dots,a_r,b_1,b_2,\dots,b_s,q,$ and $z$ be variables. Then \begin{equation}\label{2Phi1}_r\phi_s\left(\genfrac{}{}{0pt}{}{a_1,a_2,\dots,a_r}{b_1,b_2,\dots,b_s};q,z\right):=\sum_{n=0}^\infty \frac{(a_1;q)_n(a_2;q)_n\dots (a_r;q)_n}{(q;q)_n(b_1;q)_n\dots(b_s;q)_n}\left[(-1)^nq^{n\choose 2}\right]^{1-r+s}z^n.\end{equation}
Let $a$, $b$, $c$, $q$, and $z$ be variables. The $q$-Gauss summation formula \cite[II.8, p. 236]{GasperRahman} \begin{equation}\label{qGauss}
{}_2\phi_1 \left(\genfrac{}{}{0pt}{}{a,\ b}{c};q,\frac{c}{ab} \right) = \frac{(\frac{c}{a};q)_\infty (\frac{c}{b};q)_\infty}{(c;q)_\infty (\frac{c}{ab};q)_\infty},
\end{equation} and one of the three Heine's transformations \cite[III.2, p. 241]{GasperRahman}
\begin{equation}\label{Heine}
{}_2\phi_1 \left(\genfrac{}{}{0pt}{}{a,\ b}{c};q,z \right) = \frac{(\frac{c}{b};q)_\infty (bz;q)_\infty}{(c;q)_\infty(z;q)_\infty} {}_2\phi_1 \left(\genfrac{}{}{0pt}{}{\frac{abz}{c},\ b}{bz};q,\frac{c}{b} \right) 
\end{equation} will be used later on.

\section{Weighted partition identities involving G\"ollnitz--Gordon type partitions}\label{GollnitzWeights}

We start by reminding the reader of the well-known G\"ollnitz--Gordon identities of 1960's.

\begin{theorem}[Slater, 1952]\label{Analytic_GG_identities_THM} For $i\in\{1,2\}$
\begin{equation}\label{Analytic_GG_identities}
\sum_{n\geq 0} \frac{q^{n^2+2(i-1)n}(-q;q^2)_n}{(q^2;q^2)_n} = \frac{1}{(q^{2i-1};q^8)_\infty(q^4;q^8)_\infty (q^{9-2i};q^8)_\infty}.
\end{equation}
\end{theorem}

These analytic identities in \eqref{Analytic_GG_identities}, though commonly referred as G\"ollnitz--Gordon identities, were proven a decade before G\"ollnitz and Gordon by Slater \cite[(34) \& (36), p. 155]{Slater}. It should be noted that both cases of \eqref{Analytic_GG_identities} were known to Ramanujan \cite[(1.7.11--12), p. 37]{LostNotebook_2} before any known proof emerged. 

For a lot of authors, including both G\"ollnitz and Gordon, the combinatorial interpretations of \eqref{Analytic_GG_identities} have been of interest. For $i=1$ or $2$, let $\G\G_i$ be the set of partitions into parts $\geq 2i-1$ with minimal difference between parts $\geq 2$ and no consecutive even numbers appear as parts. Let $C_{i,8}$ be the set of partitions into parts congruent to $\pm (2i-1)$, and 4 mod 8. Then Theorem~\ref{Analytic_GG_identities_THM} can be rewritten in its combinatorial form \cite{Gollnitz}, \cite{Gordon}:
\begin{theorem}[G\"ollnitz--Gordon, 1967 \& 1965]\label{Combinatorial_GG_identities_THM} For $i=1$ or $2$, the number of partitions of $n$ from $\G\G_i$ is equal to the number of partitions of $n$ from $C_{i,8}$.
\begin{equation}\label{GG_combinatorial_gf}\sum_{\pi\in\G\G_i} q^{|\pi|} = \sum_{\pi\in C_{i,8}} q^{|\pi|}.\end{equation}
\end{theorem}

We now present some analytical identities that will later be interpreted in terms of the G\"ollnitz--Gordon type partitions. This discussion will yield to the first set of weighted partition identities of this paper.

\begin{theorem} \label{Analytic_Half_weight_THM}
\begin{align}\label{Analytic_Half_weight_1}\sum_{n \geq 0} \frac{q^{n^2}(-q;q^2)_n}{(q^2;q^2)^2_n} &= \frac{(-q;q^2)_\infty}{(q^2;q^2)_\infty},\\
\label{Analytic_Half_weight_2}\sum_{n \geq 0} \frac{q^{n^2+2n}(-q;q^2)_n}{(q^2;q^2)^2_n} &= \frac{(-q;q^2)_\infty}{(q^2;q^2)_\infty} \sum_{n\geq 0} \frac{(-1)^n q^{n^2+n}}{(-q;q^2)_{n+1}}=\frac{(-q;q^2)_\infty}{(q^2;q^2)_\infty}\sum_{j\geq 0} q^{3j^2+2j}(1-q^{2j+1}).
\end{align}
\end{theorem}

\begin{proof}With the definition \eqref{2Phi1}, Theorem~\ref{Analytic_Half_weight_THM} can be proven easily. The left-hand side of \eqref{Analytic_Half_weight_1} and \eqref{Analytic_Half_weight_2} can be rewritten as \[\lim_{\rho\rightarrow\infty}{}_2\phi_1 \left(\genfrac{}{}{0pt}{}{-q,\ \rho}{q^2};q^2,-\frac{q}{\rho} \right)\text{  and  }\lim_{\rho\rightarrow\infty}{}_2\phi_1 \left(\genfrac{}{}{0pt}{}{-q,\ \rho q^2}{q^2};q^2,-\frac{q}{\rho} \right),\text{ respectively.}\]
Then it is easy to show that \eqref{Analytic_Half_weight_1} is a limiting case of $q$-Gauss summation \eqref{qGauss}. An equivalent form of the identity \eqref{Analytic_Half_weight_1} is also present in Ramanujan's lost notebooks \cite[4.2.6, p. 84]{LostNotebook_2}.

The first equality of \eqref{Analytic_Half_weight_2} is an application of the Heine transformation \eqref{Heine} with $a=\rho q^2$, and the second equality is due to Fine \cite[(26.91--97), p. 62]{FineBook} with $q\mapsto -q$, and Rogers \cite[(4), p. 333]{Rogers} with $q\mapsto -q$. Another equivalent proof and an alternative representation of the second equality in \eqref{Analytic_Half_weight_2} is present in Ramanujan's lost notebooks \cite[\S 9.5]{LostNotebook_1}. 
\end{proof}

Similar to the situation in Theorem~\ref{Analytic_GG_identities_THM}, analytic identities \eqref{Analytic_Half_weight_1} and \eqref{Analytic_Half_weight_2} can be interpreted combinatorially. In fact, the interpretation of \eqref{Analytic_Half_weight_1} was discussed in \cite{Alladi_Lebesgue}. For the sake of completeness we will slightly paraphrase this discussion below. 

We can easily interpret the product on the right-hand side of \eqref{Analytic_Half_weight_1}. The expression $(-q;q^2)_\infty$ is the generating function for the number of partitions into distinct odd parts and $1/(q^2;q^2)_\infty$ is the generating function for the number of partitions into even parts. These two generating functions' product is the generating function for the number of partitions with distinct odd parts (even parts may be repeating). This is clear as the parity of a part in a partition completely identifies which generating function it is coming from.

The generating function interpretation of the left-hand side of \eqref{Analytic_Half_weight_1} needs us to identify weights on partitions. For a positive integer $n$, the reciprocal of the $q$-Pochhammer symbol \begin{equation}\label{even_pts_less_n}\frac{1}{(q^2;q^2)_n}\end{equation} is the generating function for the number of partitions into $\leq n$ even parts. The expression 
\begin{equation}\label{inner_GG_GF}\frac{q^{n^2}(-q;q^2)_n}{(q^2;q^2)_n}\end{equation} 
can be interpreted as the generating function for the number of partitions into exactly $n$ parts from $\G\G_1$, \cite[(8.2), p. 173]{AlladiGollnitz}. We can represent the partitions counted by \eqref{inner_GG_GF} as 2-modular graphs. There are four possible patterns that can appear at the end of consecutive parts of these 2-modular Ferrers diagrams. All these possible endings of consecutive parts $\lambda_i$ and $\lambda_{i+1}$, where $i<\nu(\pi)$ of a partition $\pi$, are demonstrated in Table~\ref{Tips of parts}.

\begin{center}
\begin{table}[htb]\caption{Ends of consecutive parts of G\"ollnitz-Gordon partitions}\label{Tips of parts}
\begin{tabular}{cc}
\begin{tikzpicture}[line cap=round,line join=round,>=triangle 45,x=0.16666cm,y=0.16666cm]
\clip(-4,0.33) rectangle (32,11);
\draw (1,7)-- (4,7);
\draw (4,7)-- (4,1);
\draw (4,1)-- (1,1);
\draw (1,1)-- (1,7);
\draw (1,4)-- (4,4);
\draw (9,4)-- (12,4);
\draw (12,4)-- (12,7);
\draw (12,7)-- (9,7);
\draw (9,7)-- (9,4);
\draw (9,4)-- (9,1);
\draw (9,1)-- (12,1);
\draw (12,4)-- (12,1);
\draw (13,4)-- (16,4);
\draw (16,4)-- (16,7);
\draw (16,7)-- (13,7);
\draw (13,7)-- (13,4);
\draw (21,7)-- (21,4);
\draw (21,4)-- (24,4);
\draw (24,4)-- (24,7);
\draw (24,7)-- (21,7);
\draw (25,7)-- (28,7);
\draw (28,7)-- (28,4);
\draw (28,4)-- (25,4);
\draw (25,4)-- (25,7);
\draw (2.5,5.5) node[anchor=center] {2};
\draw (2.5,2.5) node[anchor=center] {2};
\draw (10.5,5.5) node[anchor=center] {2};
\draw (14.5,5.5) node[anchor=center] {2};
\draw (22.5,5.5) node[anchor=center] {2};
\draw (26.5,5.5) node[anchor=center] {1};
\draw (10.5,2.5) node[anchor=center] {1};
\draw (6.5,4) node[anchor=center] {$ \dots $};
\draw (19,5.5) node[anchor=center] {$ \dots $};
\draw (12.66,7.68)-- (12.66,8.35);
\draw (24.33,7.66)-- (24.33,8.33);
\draw (12.66,8)-- (24.33,8.01);
\draw (18.4,9.5)  node[anchor=center] {$ \frac{\lambda_i - \lambda_{i+1}-2}{2} \geq 0 $};
\draw (-1,5.5) node[anchor=center] {$ \lambda_i $};
\draw (-1,2.5) node[anchor=center] {$ \lambda_{i+1} $};
\draw (28,7)-- (28,4);
\draw (28,4)-- (28,7);
\end{tikzpicture},&
\begin{tikzpicture}[line cap=round,line join=round,>=triangle 45,x=0.16666cm,y=0.16666cm]
\clip(-4,0.33) rectangle (32.33,11);
\draw (1,7)-- (4,7);
\draw (4,7)-- (4,1);
\draw (4,1)-- (1,1);
\draw (1,1)-- (1,7);
\draw (1,4)-- (4,4);
\draw (9,4)-- (12,4);
\draw (12,4)-- (12,7);
\draw (12,7)-- (9,7);
\draw (9,7)-- (9,4);
\draw (9,4)-- (9,1);
\draw (9,1)-- (12,1);
\draw (12,4)-- (12,1);
\draw (13,4)-- (16,4);
\draw (16,4)-- (16,7);
\draw (16,7)-- (13,7);
\draw (13,7)-- (13,4);
\draw (21,7)-- (21,4);
\draw (21,4)-- (24,4);
\draw (24,4)-- (24,7);
\draw (24,7)-- (21,7);
\draw (25,7)-- (28,7);
\draw (28,7)-- (28,4);
\draw (28,4)-- (25,4);
\draw (25,4)-- (25,7);
\draw (2.5,5.5) node[anchor=center] {2};
\draw (2.5,2.5) node[anchor=center] {2};
\draw (10.5,5.5) node[anchor=center] {2};
\draw (14.5,5.5) node[anchor=center] {2};
\draw (22.5,5.5) node[anchor=center] {2};
\draw (26.5,5.5) node[anchor=center] {2};
\draw (10.5,2.5) node[anchor=center] {2};
\draw (6.5,4) node[anchor=center] {$ \dots $};
\draw (19,5.5) node[anchor=center] {$ \dots $};
\draw (12.66,7.68)-- (12.66,8.35);
\draw (24.33,7.66)-- (24.33,8.33);
\draw (12.66,8)-- (24.33,8.01);
\draw (18.4,9.5) node[anchor=center] {$ \frac{\lambda_i - \lambda_{i+1}-3}{2}\geq 0  $};
\draw (-1,5.5) node[anchor=center] {$ \lambda_i $};
\draw (-1,2.5) node[anchor=center] {$ \lambda_{i+1} $};
\draw (31,7)-- (31,4);
\draw (31,4)-- (28,4);
\draw (31,7)-- (28,7);
\draw (28,7)-- (28,4);
\draw (28,4)-- (28,7);
\draw (29.5,5.5) node[anchor=center] {1};
\end{tikzpicture},\\
\begin{tikzpicture}[line cap=round,line join=round,>=triangle 45,x=0.16666cm,y=0.16666cm]
\clip(-4,0.33) rectangle (30,11);
\draw (1,7)-- (4,7);
\draw (4,7)-- (4,1);
\draw (4,1)-- (1,1);
\draw (1,1)-- (1,7);
\draw (1,4)-- (4,4);
\draw (9,4)-- (12,4);
\draw (12,4)-- (12,7);
\draw (12,7)-- (9,7);
\draw (9,7)-- (9,4);
\draw (9,4)-- (9,1);
\draw (9,1)-- (12,1);
\draw (12,4)-- (12,1);
\draw (13,4)-- (16,4);
\draw (16,4)-- (16,7);
\draw (16,7)-- (13,7);
\draw (13,7)-- (13,4);
\draw (21,7)-- (21,4);
\draw (21,4)-- (24,4);
\draw (24,4)-- (24,7);
\draw (24,7)-- (21,7);
\draw (25,7)-- (28,7);
\draw (28,7)-- (28,4);
\draw (28,4)-- (25,4);
\draw (25,4)-- (25,7);
\draw (2.5,5.5) node[anchor=center] {2};
\draw (2.5,2.5) node[anchor=center] {2};
\draw (10.5,5.5) node[anchor=center] {2};
\draw (14.5,5.5) node[anchor=center] {2};
\draw (22.5,5.5) node[anchor=center] {2};
\draw (26.5,5.5) node[anchor=center] {2};
\draw (10.5,2.5) node[anchor=center] {1};
\draw (6.5,4) node[anchor=center] {$ \dots $};
\draw (19,5.5) node[anchor=center] {$ \dots $};
\draw (12.66,7.68)-- (12.66,8.35);
\draw (24.33,7.66)-- (24.33,8.33);
\draw (12.66,8)-- (24.33,8.01);
\draw (18.4,9.5) node[anchor=center] {$ \frac{\lambda_i - \lambda_{i+1}-3}{2} \geq 0 $};
\draw (-1,5.5) node[anchor=center] {$ \lambda_i $};
\draw (-1,2.5) node[anchor=center] {$ \lambda_{i+1} $};
\draw (28,7)-- (28,4);
\draw (28,4)-- (28,7);
\end{tikzpicture},&
\begin{tikzpicture}[line cap=round,line join=round,>=triangle 45,x=0.16666cm,y=0.16666cm]
\clip(-4,0.33) rectangle (32,11);
\draw (1,7)-- (4,7);
\draw (4,7)-- (4,1);
\draw (4,1)-- (1,1);
\draw (1,1)-- (1,7);
\draw (1,4)-- (4,4);
\draw (9,4)-- (12,4);
\draw (12,4)-- (12,7);
\draw (12,7)-- (9,7);
\draw (9,7)-- (9,4);
\draw (9,4)-- (9,1);
\draw (9,1)-- (12,1);
\draw (12,4)-- (12,1);
\draw (13,4)-- (16,4);
\draw (16,4)-- (16,7);
\draw (16,7)-- (13,7);
\draw (13,7)-- (13,4);
\draw (21,7)-- (21,4);
\draw (21,4)-- (24,4);
\draw (24,4)-- (24,7);
\draw (24,7)-- (21,7);
\draw (25,7)-- (28,7);
\draw (28,7)-- (28,4);
\draw (28,4)-- (25,4);
\draw (25,4)-- (25,7);
\draw (2.5,5.5) node[anchor=center] {2};
\draw (2.5,2.5) node[anchor=center] {2};
\draw (10.5,5.5) node[anchor=center] {2};
\draw (14.5,5.5) node[anchor=center] {2};
\draw (22.5,5.5) node[anchor=center] {2};
\draw (26.5,5.5) node[anchor=center] {2};
\draw (10.5,2.5) node[anchor=center] {2};
\draw (6.5,4) node[anchor=center] {$ \dots $};
\draw (19,5.5) node[anchor=center] {$ \dots $};
\draw (12.66,7.68)-- (12.66,8.35);
\draw (24.33,7.66)-- (24.33,8.33);
\draw (12.66,8)-- (24.33,8.01);
\draw (18.4,9.5)  node[anchor=center] {$ \frac{\lambda_i - \lambda_{i+1}-4}{2}\geq 0  $};
\draw (-1,5.5) node[anchor=center] {$ \lambda_i $};
\draw (-1,2.5) node[anchor=center] {$ \lambda_{i+1} $};
\draw (31,7)-- (31,4);
\draw (31,4)-- (28,4);
\draw (31,7)-- (28,7);
\draw (28,7)-- (28,4);
\draw (28,4)-- (28,7);
\draw (29.5,5.5) node[anchor=center] {2};
\end{tikzpicture}.
\end{tabular}
\end{table}
\end{center}

The labelled gaps on Table~\ref{Tips of parts} are the number of non-essential number of boxes between consecutive parts for the partition to be in $\G\G_1$. These differences can be equal to zero. In general, the number of the non-essential boxes of the 2-modular Ferrers diagram for a partition in $\G\G_1$ is given by the formula \begin{equation}\label{Num_of_non_essential_boxes}
\frac{\lambda_i - \lambda_{i+1}-\delta_{\lambda_i,e}-\delta_{\lambda_{i+1},e}}{2} - 1,
\end{equation}
where \begin{equation}\label{delta_even}
\delta_{n,e} := \left\{ \begin{array}{ll} 1, & \text{if }n\text{ is even,}\\ 0, & \text{otherwise.}\end{array} \right. 
\end{equation} Later we will need \begin{equation}\label{delta_odd}
\delta_{n,o} := 1-\delta_{n,e}. 
\end{equation}

The product of \eqref{even_pts_less_n} and \eqref{inner_GG_GF} is the generating function for the number of partitions from $\G\G_1$ into exactly $n$ parts where the non-essential boxes in their 2-modular Ferrers diagram representation come in two colors. Let $\pi'$ be a partition counted by \eqref{even_pts_less_n} in one color and let $\pi^*$ be a partition counted by \eqref{inner_GG_GF} in another color. Then we insert columns of the 2-modular Ferrers diagram representation of $\pi'$ in the 2-modular Ferrers diagram representation of $\pi^*$. In doing so, we insert the different colored columns all the way left those columns can be inserted, without violating the definition of 2-modular Ferrers diagrams. One example of the insertion of this type for $n= 4$ is presented in Table~\ref{Table_insertion example}.

\begin{table}[htb]\caption{Insertion of the columns of $\pi'=(6,4,4,4)$ in $\pi^* = (12,8,3,1)\in \G\G_1$.}\label{Table_insertion example}
\begin{center}
\definecolor{ffqqqq}{rgb}{0,1,0}
\definecolor{qqqqff}{rgb}{1,1,1}
\definecolor{cqcqcq}{rgb}{0.75,0.75,0.75}
\begin{tikzpicture}[line cap=round,line join=round,>=triangle 45,x=0.5cm,y=0.5cm]
\clip(1.5,1.5) rectangle (24.5,7.3);
\fill[line width=0.pt,color=ffqqqq,fill=ffqqqq,fill opacity=0.4] (10.,6.) -- (13.,6.) -- (13.,5.) -- (12.,5.) -- (12.,2.) -- (10.,2.) -- cycle;
\fill[line width=0.pt,color=ffqqqq,fill=ffqqqq,fill opacity=0.4] (15.,6.) -- (17.,6.) -- (17.,2.) -- (15.,2.) -- cycle;
\fill[line width=0.pt,color=ffqqqq,fill=ffqqqq,fill opacity=0.4] (21.,6.) -- (22.,6.) -- (22.,5.) -- (21.,5.) -- cycle;
\draw [line width=1pt] (3,2)-- (3,6);
\draw [line width=1pt] (2,6)-- (2,2);
\draw [line width=1pt] (2,2)-- (3,2);
\draw [line width=1pt] (2,3)-- (4,3);
\draw [line width=1pt] (4,3)-- (4,6);
\draw [line width=1pt] (6,4)-- (2,4);
\draw [line width=1pt] (6,4)-- (6,6);
\draw [line width=1pt] (8,6)-- (8,5);
\draw [line width=1pt] (7,6)-- (7,5);
\draw [line width=1pt] (5,6)-- (5,4);
\draw [line width=1pt] (10,6)-- (10,2);
\draw [line width=1pt] (10,2)-- (12,2);
\draw [line width=1pt] (10,6)-- (13,6);
\draw [line width=1pt] (13,6)-- (13,5);
\draw [line width=1pt] (13,5)-- (10,5);
\draw [line width=1pt] (12,2)-- (12,6);
\draw [line width=1pt] (11,2)-- (11,6);
\draw [line width=1pt] (12,4)-- (10,4);
\draw [line width=1pt] (12,3)-- (10,3);
\draw [line width=1pt] (2,6)-- (8,6);
\draw [line width=1pt] (8,5)-- (2,5);
\draw [line width=1pt] (15,2)-- (18,2);
\draw [line width=1pt] (18,6)-- (18,2);
\draw [line width=1pt] (15,6)-- (15,2);
\draw [line width=1pt] (15,6)-- (24,6);
\draw [line width=1pt] (24,6)-- (24,5);
\draw [line width=1pt] (24,5)-- (15,5);
\draw [line width=1pt] (21,4)-- (21,6);
\draw [line width=1pt] (21,4)-- (15,4);
\draw [line width=1pt] (19,3)-- (19,6);
\draw [line width=1pt] (19,3)-- (15,3);
\draw [line width=1pt] (16,6)-- (16,2);
\draw [line width=1pt] (17,6)-- (17,2);
\draw [line width=1pt] (20,6)-- (20,4);
\draw [line width=1pt] (22,6)-- (22,5);
\draw [line width=1pt] (23,5)-- (23,6);
\draw (2.5,5.5) node[anchor=center] {2};
\draw (3.5,5.5) node[anchor=center] {2};
\draw (4.5,5.5) node[anchor=center] {2};
\draw (5.5,5.5) node[anchor=center] {2};
\draw (5,6.6) node[anchor=center] {$\pi^*$};
\draw (6.5,5.5) node[anchor=center] {2};
\draw (7.5,5.5) node[anchor=center] {2};
\draw (2.5,4.5) node[anchor=center] {2};
\draw (3.5,4.5) node[anchor=center] {2};
\draw (4.5,4.5) node[anchor=center] {2};
\draw (5.5,4.5) node[anchor=center] {2};
\draw (2.5,3.5) node[anchor=center] {2};
\draw (3.5,3.5) node[anchor=center] {1};
\draw (10.5,5.5) node[anchor=center] {2};
\draw (11.5,6.6) node[anchor=center] {$\pi'$};
\draw (11.5,5.5) node[anchor=center] {2};
\draw (12.5,5.5) node[anchor=center] {2};
\draw (10.5,4.5) node[anchor=center] {2};
\draw (11.5,4.5) node[anchor=center] {2};
\draw (10.5,3.5) node[anchor=center] {2};
\draw (11.5,3.5) node[anchor=center] {2};
\draw (10.5,2.5) node[anchor=center] {2};
\draw (11.5,2.5) node[anchor=center] {2};
\draw (15.5,5.5) node[anchor=center] {2};
\draw (16.5,5.5) node[anchor=center] {2};
\draw (15.5,4.5) node[anchor=center] {2};
\draw (16.5,4.5) node[anchor=center] {2};
\draw (15.5,3.5) node[anchor=center] {2};
\draw (16.5,3.5) node[anchor=center] {2};
\draw (15.5,2.5) node[anchor=center] {2};
\draw (16.5,2.5) node[anchor=center] {2};
\draw (17.5,3.5) node[anchor=center] {2};
\draw (18.5,3.5) node[anchor=center] {1};
\draw (20.5,4.5) node[anchor=center] {2};
\draw (19.5,4.5) node[anchor=center] {2};
\draw (20.5,5.5) node[anchor=center] {2};
\draw (19.5,5.5) node[anchor=center] {2};
\draw (18.5,5.5) node[anchor=center] {2};
\draw (18.5,4.5) node[anchor=center] {2};
\draw (17.5,4.5) node[anchor=center] {2};
\draw (17.5,5.5) node[anchor=center] {2};
\draw (21.5,5.5) node[anchor=center] {2};
\draw (22.5,5.5) node[anchor=center] {2};
\draw (23.5,5.5) node[anchor=center] {2};
\draw (2.5,2.5) node[anchor=center] {1};
\draw (17.5,2.5) node[anchor=center] {1};
\draw (9,4) node[anchor=center] {,};
\draw (14,4) node[anchor=center] {$ \rightarrow $};
\end{tikzpicture}\\
The inserted columns from $\pi'$ are all non-essential for the outcome partition to lie in $\G\G_1$, though those are not the only non-essential columns.
\end{center}
\end{table}

This insertion changes the number of non-essential boxes of a partition  $\pi\in\G\G_1$, and it does not effect any essential structure of the 2-modular Ferrers diagrams. There are a total of \begin{equation}\label{Raw_half_weight}
\frac{\lambda_i - \lambda_{i+1}-\delta_{\lambda_i,e}-\delta_{\lambda_{i+1},e}}{2} 
\end{equation} many different possibilities for the coloration of the non-essential boxes that appear from the part $\lambda_i$ to $\lambda_{i+1}$. Similarly, there are \[
\frac{\lambda_{n}+\delta_{\lambda_{n},o}}{2}
\] many coloration possibilities for a the smallest part of a partition that gets counted by the summand of \eqref{Analytic_Half_weight_1}, where $\delta_{n,o}$ is defined in \eqref{delta_odd}.

Hence, combining all the possible number of colorations, there are \begin{equation}\label{Half_Weight_1}
\omega_1(\pi) := \frac{\lambda_{\nu(\pi)}+\delta_{\lambda_{\nu(\pi)},o}}{2}\cdot\prod_{i=1}^{\nu(\pi)-1} \frac{\lambda_i - \lambda_{i+1}-\delta_{\lambda_i,e}-\delta_{\lambda_{i+1},e}}{2}
\end{equation} total number of colorations of a partition $\pi\in\G\G_1$. The far right partition in Table~\ref{Table_insertion example} is one of the possible colorations of the partition $(18,12,7,5)$, and the total number of colorations via \eqref{Half_Weight_1} is $\omega_1(18,12,7,5) = 3\cdot2\cdot 2\cdot 1= 12$.

The above discussion yields the weighted partition identity:
\begin{theorem}[Alladi, 2012]\label{GG_1_combinatorial_theorem} \begin{equation}
\sum_{\pi\in\G\G_1} \omega_1(\pi) q^{|\pi|} = \sum_{\pi\in\P_{do}} q^{|\pi|},
\end{equation} where $\omega_1$ is defined as in \eqref{Half_Weight_1} and $\P_{do}$ is the set of partitions with distinct odd parts.
\end{theorem}

Theorem~\ref{GG_1_combinatorial_theorem} is essentially \cite[Theorem~3]{Alladi_Lebesgue} with $a=b=1$ with minor corrections for the weight associated with the smallest part of G\"ollnitz--Gordon partitions. The set of partitions $\P_{do}$, partitions with distinct odd parts, has also been studied in \cite{AlladiAnnalsOfComb} and \cite{BerkovichGarvanCrank}. We give an example of Theorem~\ref{GG_1_combinatorial_theorem} in Table~\ref{GG1_table}.

\begin{table}[htb]
\caption{Example of Theorem~\ref{GG_1_combinatorial_theorem} with $|\pi|=12$.}\label{GG1_table}
\begin{center}
$\begin{array}{cc||c|c}
\pi\in\G\G_1	& \omega_1 	& \pi\in\P_{do} 	& \pi\in\P_{do}\\
[-2ex]& & &\\
(12)						&	6		&			(12)			&		(6,4,2)			\\
(11,1)						&	5		&			(11,1)			&		(6,3,2,1)		\\
(10,2)						&	3		&			(10,2)	 		&		(6,2,2,2)		\\
(9,3)						&	6		&			(9,3)			&		(5,4,3)		\\
(8,4)						&	2		&			(9,2,1)			&		(5,4,2,1)		\\
(8,3,1)						&	2		&			(8,4)			&		(5,3,2,2)		\\
(7,5)						&	3		&			(8,3,1)			&		(5,2,2,2,1)		\\
(7,4,1)						&	1		&			(8,2,2)			&		(4,4,4)		\\
							&			&			(7,5)			&		(4,4,3,1)	\\
							&			&			(7,4,1)			&		(4,4,2,2)			\\
							&			&			(7,3,2)			&		(4,3,2,2,1)		\\		
							&			&			(7,2,2,1)		&		(4,2,2,2,2)		\\		
							&			&			(6,6)			&		(3,2,2,2,2,1)		\\
							&			&			(6,5,1)			&		(2,2,2,2,2,2)		\\		\end{array}$
		
\vspace{2mm}					
The summation of all $\omega_{1}(\pi)$ values for $\pi\in \G\G_1$ with $|\pi|=12$ equals $28$ as the number of partitions from $\P_{do}$ with the same norm.
\end{center}
\end{table}

Formally, let $\lambda_{\nu(\pi)+1} := 0$ for a partition $\pi$. Following the same construction \eqref{even_pts_less_n}--\eqref{Half_Weight_1} step-by-step for the $\G\G_2$ type partitions, we see that the left-hand side of \eqref{Analytic_Half_weight_2} can be interpreted as a weighted generating function for the number of partitions \[\sum_{\pi\in\G\G_2} \omega_2(\pi) q^{|\pi|}, \] where \begin{equation}\label{weight_2}\omega_2(\pi) := \prod_{i=1}^{\nu(\pi)} \frac{\lambda_i - \lambda_{i+1}-\delta_{\lambda_i,e}-\delta_{\lambda_{i+1},e}}{2}.\end{equation} 
This weight $\omega_2$, unlike $\omega_1$, is uniform on every pair of consecutive parts with our customary definition $\lambda_{\nu(\pi)+1} = 0$.

We rewrite the sum in the middle term of \eqref{Analytic_Half_weight_2} as \begin{equation}\label{EvenODDSplit}\frac{(-q;q^2)_\infty}{(q^2;q^2)_\infty}\sum_{n\geq 0} \frac{(-1)^n q^{n^2+n}}{(-q;q^2)_{n+1}} = \frac{(-q;q^2)_\infty}{(q^2;q^2)_\infty}\left(\sum_{j\geq 0} \frac{q^{4j^2 +2j}(1-q^{4j+2})}{(-q;q^2)_{2j+1}} + \sum_{j\geq 1} \frac{q^{4j^2 + 2j-1}}{(-q;q^2)_{2j}}\right).\end{equation} 
Clearly \eqref{EvenODDSplit} amounts to 
\begin{align}
\label{EQN1}\sum_{j\geq 0} \frac{q^{4j^2+2j}}{(-q;q^2)_{2j+1}} - \sum_{j\geq 1} \frac{q^{4j^2-2j}}{(-q;q^2)_{2j}} &=\sum_{j\geq 0} \frac{q^{4j^2 +2j}(1-q^{4j+2})}{(-q;q^2)_{2j+1}} + \sum_{j\geq 1} \frac{q^{4j^2 + 2j-1}}{(-q;q^2)_{2j}},\\
\intertext{where we split the sum on the left of \eqref{EvenODDSplit} into two sub-sums according to the parity of the summation variable and changing the variable name $n$ to $j$. After cancellations, \eqref{EQN1} turns into}
\label{EQN2}\sum_{j\geq 1} \frac{q^{4j^2-2j}(1+q^{4j-1})}{(-q;q^2)_{2j}} &=\sum_{j\geq 0} \frac{q^{4j^2 +6j+2}}{(-q;q^2)_{2j+1}}.
\end{align}
The equation \eqref{EQN2} can be easily established by simplifying the fraction on the left and shifting the summation variable $j\mapsto j+1$. 

Let $\P_{rdo}$ be the set of partitions with distinct odd parts with the additional restrictions that the smallest part is $>1$, and if the smallest part of a partition $\pi$ is $2$, then $\pi$ starts either as \begin{align}
\label{dist1} \pi &= (2^{f_2},4^{f_4},6^{f_6},\dots, (4j-2)^{f_{4j-2}}, (4j-1)^{1},\dots),\\ \intertext{where $f_2,\ f_4,\ \dots,\ f_{4j-2}$ are all positive, or as}
\label{dist3} \pi &= (2^{f_2},4^{f_4},6^{f_6},\dots, (4j)^{f_{4j}}, (4j+1)^{0},(4j+2)^{0},\dots),
 \end{align} where $f_2,\ f_4,\ \dots,\ f_{4j}$ are all positive, for any positive $j$. We now claim that the middle term of \eqref{Analytic_Half_weight_2} is the generating function for the number the partitions from the set $\P_{rdo}$. We demonstrate this with the aid of \eqref{EvenODDSplit}. Using distribution on the right of \eqref{EvenODDSplit}, we get
\begin{align}
\nonumber \frac{(-q;q^2)_\infty}{(q^2;q^2)_\infty}&\left(\sum_{j\geq 0} \frac{q^{4j^2 +2j}(1-q^{4j+2})}{(-q;q^2)_{2j+1}} + \sum_{j\geq 1} \frac{q^{4j^2 + 2j-1}}{(-q;q^2)_{2j}}\right)\\\nonumber&= \sum_{j\geq 0} q^{4j^2 +2j}(-q^{4j+3};q^2)_\infty\frac{1-q^{4j+2}}{(q^2;q^2)_\infty} + \sum_{j\geq 1} q^{4j^2 + 2j-1} \frac{(-q^{4j+1};q^2)_\infty}{(q^2;q^2)_{\infty}}.\\ \intertext{Thus we have for the right-hand side of \eqref{EvenODDSplit}}
\label{even_odd_to_Prdwo} \frac{(-q^3;q^2)_\infty}{(q^4;q^2)_\infty}+\sum_{j\geq 1} q^{4j^2 + 2j-1}& \frac{(-q^{4j+1};q^2)_\infty}{(q^2;q^2)_{\infty}}+\sum_{j\geq 1} q^{4j^2 +2j}(-q^{4j+3};q^2)_\infty\frac{1-q^{4j+2}}{(q^2;q^2)_\infty}. 
\end{align} 
For positive integers $j$, we can write \begin{align*} 4j^2 + 2j-1 &= 2+4+6+\dots+(4j-2)+(4j-1),\\ 4j^2 +2j &= 2+4+\dots+4j.\end{align*} The above implies the initial conditions in \eqref{dist1} and \eqref{dist3}, respectively. The presence of the distinct odd parts and the (possibly repeated) even parts is clear from the shifted $q$-factorials. This proves,

\begin{theorem}\label{Half_Weight_Theorem_2}
\begin{equation}\label{Half_Weight_Theorem_2_eqn}
\sum_{\pi\in\G\G_2} \omega_2(\pi) q^{|\pi|} = \sum_{\pi\in\P_{rdo}} q^{|\pi|},
\end{equation}
where $\omega_2$ as in \eqref{weight_2}.
\end{theorem}

The second equality in \eqref{Analytic_Half_weight_2} connects an order 3 false theta function with the combinatorial objects we have interpreted above \eqref{Half_Weight_Theorem_2_eqn}. This false theta function can also be interpreted as a generating function for the number of partitions on a set after some modification. It is easy to see that
\begin{equation}\label{False_Theta_Half_weight}\frac{(-q;q^2)_\infty}{(q^2;q^2)_\infty}\sum_{j\geq 0} q^{3j^2+2j}(1-q^{2j+1}) = \sum_{j\geq 0} q^{3j^2+2j}\frac{(-q;q^2)_{j}}{(q^2;q^2)_{2j}}\frac{(-q^{2j+3};q^2)_\infty}{(q^{4j+4};q^2)_\infty}.\end{equation}
Let $\mathcal{A}$ denote the set of partitions where for any partition \begin{enumerate}[i.]\item the first integer that is not a part is odd, \item the double of the first missing part is also missing, \item each even part less than the first missing part appears at least twice, \item each odd part less than the first missing part appears at most twice,  \item each odd larger than the first missing part is not repeated.\end{enumerate} The expression \eqref{False_Theta_Half_weight} can be interpreted ---as G. E. Andrews did \cite{PrivateConv}--- as the generating function for the number of partitions from the set $\mathcal{A}$. This interpretation can be seen after the clarification that $3j^2+2j = 1+ 2+ 2+ 3+ 4+4+\dots +(2j-1)+2j+2j$. 

The set used in this interpretation does not consist of distinct odd parts necessarily, and therefore gets out of the scope of the identity of Theorem~\ref{Half_Weight_Theorem_2}. Nevertheless, this observation finalizes the discussion of the combinatorial version of \eqref{Analytic_Half_weight_2}:
\begin{theorem}\label{Half_Weight_Theorem_Most_General}
\begin{equation*}
\sum_{\pi\in\G\G_2} \omega_2(\pi) q^{|\pi|} = \sum_{\pi\in\P_{rdo}} q^{|\pi|} = \sum_{\pi\in\mathcal{A}} q^{|\pi|}.
\end{equation*}where $\omega_2$ as in \eqref{weight_2}.
\end{theorem}

We demonstrate Theorem~\ref{Half_Weight_Theorem_Most_General} in Table~\ref{Table_Example_Half_weight_2}.

\begin{table}[htb]
\caption{Example of Theorem~\ref{Half_Weight_Theorem_Most_General} with $|\pi|=12$.} \label{Table_Example_Half_weight_2}
\begin{center}
$\begin{array}{cc||c||c}
\pi\in\G\G_2& \omega_2(\pi)&{\pi\in\P_{rdo}}&\pi\in\mathcal{A}\\
[-2ex]& & &\\
(12) & 5 & (12) & (12)\\
(9,3) & 3 &(9,3) & (9,3)\\
(8,4) & 1 &(8,4) & (8,4)\\
(7,5)& 2 & (7,5) & (7,5)\\
& & (7,3,2) & (7,2,2,1)\\
& & (6,6) & (6,6)\\
& & (5,4,3) & (5,4,3)\\
& & (5,3,2,2) & (5,2,2,2,1)\\
& & (4,4,4) & (4,4,4)\\
& & (4,4,2,2) & (4,2,2,2,1,1)\\
& & (4,2,2,2,2) & (2,2,2,2,2,1,1)\\
\end{array}$

\vspace{2mm}
The summation of all $\omega_{2}(\pi)$ values for $\pi\in \G\G_2$ with $|\pi|=12$ equals $11$ as the number of partitions from $\P_{rdo}$ and $\mathcal{A}$ with the same norm.
\end{center}
\end{table}

Recall that $\R\R$ is the set of partitions into distinct parts with difference between parts $\geq 2$. We also note that, similar to \eqref{RR_to_D}, the choice of the set $\G\G_2$ in Theorem~\ref{Half_Weight_Theorem_Most_General} can be replaced with a superset such as $\G\G_1$ or $\R\R$. The weight $\omega_2(\pi)$ would vanish for a partition $\pi\in\R\R \setminus \G\G_2$. In particular, we have \[\sum_{\pi\in\G\G_2} \omega_2(\pi) q^{|\pi|} = \sum_{\pi\in \R\R} \omega_2(\pi) q^{|\pi|}.\]

\section{Weighted partition identities relating partitions into distinct parts and unrestricted partitions}\label{Overpartitions_Section}

We start with two identities that will yield weighted partition identities between the sets $\mathcal{D}$, partitions into distinct parts, and $\mathcal{U}$, the set of all partitions.

\begin{theorem} \label{Analytic_Double_weight_THM}
\begin{align}\label{Analytic_Double_weight_1}\sum_{n \geq 0} \frac{q^{{(n^2+n)}/{2}}(-1;q)_n}{(q;q)^2_n} &= \frac{(-q;q)_\infty}{(q;q)_\infty},\\
\label{Analytic_Double_weight_2}\sum_{n \geq 0} \frac{q^{(n^2+n)/{2}}(-q;q)_n}{(q;q)^2_n} &= \frac{(-q;q)_\infty}{(q;q)_\infty} \left(1+\sum_{n\geq 1} \frac{(-1)^n q^{2n-1}}{(-q;q^2)_{n}}\right)\\ \label{Analytic_Double_weight_False_Theta}&=\frac{(-q;q)_\infty}{(q;q)_\infty}\sum_{j\geq 0} q^{{(3j^2+j)}/{2}}(1-q^{2j+1}).
\end{align}
\end{theorem}

\begin{proof}We note that the left-hand sides of \eqref{Analytic_Double_weight_1} and \eqref{Analytic_Double_weight_2} are \[\lim_{\rho\rightarrow\infty}{}_2\phi_1 \left(\genfrac{}{}{0pt}{}{-1,\ \rho}{q};q,-\frac{q}{\rho} \right)\text{  and  }\lim_{b\rightarrow -1}\lim_{\rho\rightarrow\infty}{}_2\phi_1 \left(\genfrac{}{}{0pt}{}{\rho,\ qb }{qb^2};q,\frac{qb}{\rho} \right),\text{ respectively.}\] Similar to the case of Theorem~\ref{Analytic_Half_weight_THM}, equation \eqref{Analytic_Double_weight_1} is a special case of the $q$-Gauss identity \eqref{qGauss}. This identity has also been previously proven in the work of Starcher \cite[(3.7), p. 805]{Starcher}.

Identity \eqref{Analytic_Double_weight_2} is more involved. To establish the equality of \eqref{Analytic_Double_weight_2}, we apply the Heine transformation \eqref{Heine} with $a=\rho$ which yields
\begin{equation}\label{OneStepAwayFromF}
\sum_{n \geq 0} \frac{q^{(n^2+n)/{2}}(-q;q)_n}{(q;q)^2_n} = \lim_{b\rightarrow-1}\lim_{\rho\rightarrow\infty} \frac{(b;q)_\infty(q^2b^2/\rho;q)_\infty}{(qb^2;q)_\infty (qb/\rho;q)_\infty} \sum_{n\geq 0} \frac{(qb;q)_n}{(q^2b^2/\rho;q)_n} b^n.
\end{equation}
After the limit $\rho\rightarrow\infty$, the sum on the right of \eqref{OneStepAwayFromF} turns into
\begin{equation}\label{2F(-1,0,-1)}
\sum_{n \geq 0} \frac{q^{(n^2+n)/{2}}(-q;q)_n}{(q;q)^2_n} = \lim_{b\rightarrow-1}\frac{(bq;q)_\infty}{(qb^2;q)_\infty} (1-b)F(b,0;b),
\end{equation} where in Fine's notation \cite[(1.1)]{FineBook} \[F(a,b;t) := {}_2\phi_1 \left(\genfrac{}{}{0pt}{}{q,\ aq}{bq};q,t \right).\]
We have three explicit formulas for the expression $\lim_{b\rightarrow-1}(1-b)F(b,0;b)$ coming from Fine's work:
\begin{align}
\label{Fine2}\lim_{b\rightarrow-1}(1-b)F(b,0;b)&= \sum_{j\geq 0} q^{{(3j^2+j)}/{2}}(1-q^{2j+1})\\
\label{Fine1} &=1 + \sum_{n\geq 1} \frac{(-1)^n q^{2n-1}}{(-q;q^2)_n}\\
\label{Fine3} &= \sum_{n=0}^\infty \frac{(-1)^n q^{(n^2+n)/2}}{(-q;q)_n}.
\end{align}
These identities are \cite[(7.7), p. 7]{FineBook}, \cite[(23.2), p. 45]{FineBook}, and \cite[(6.1), p. 4]{FineBook} with $a=t\rightarrow-1$, respectively. Formulas \eqref{Fine2} and \eqref{Fine1} in comparison with \eqref{2F(-1,0,-1)} prove both \eqref{Analytic_Double_weight_False_Theta} and \eqref{Analytic_Double_weight_2}, respectively. 
\end{proof}

Note that the equality of the right sides of the identities \eqref{Fine2}--\eqref{Fine3} can be proved in a purely combinatorial  manner  with the aid of Sylvester's bijection \cite{BressoudBook} and Franklin's involution \cite{Theory_of_Partitions}. The equality of \eqref{Fine2} and \eqref{Fine3} will be used later in the proof of the Theorem~\ref{RestrictedTHM}.

We remark that identity \eqref{Analytic_Double_weight_1} was further studied in \cite{Overpartitions_Paper}. There the identity was combinatorially interpreted as a relation between generalized Frobenius symbols and overpartitions. 

Now we will move on to our discussion of combinatorial interpretations of the analytic identities of Theorem~\ref{Analytic_Double_weight_THM}. We have already pointed out that the product on the right side \eqref{Analytic_Double_weight_1} is a special case of Alladi's \eqref{overpartitions} with $a=1$, $b=2$ and $n\rightarrow\infty$. This can be interpreted as the weighted sum on the set of partitions $\mathcal{U}$: \begin{equation*}\tag{\ref{weights_of_overpartitions}}
\frac{(-q;q)_\infty}{(q;q)_\infty} = \sum_{\pi \in \mathcal{U}} 2^{\nu_d(\pi)}q^{|\pi|},
\end{equation*} where $\nu_d(\pi)$ is the number of different parts of $\pi$. 

The left-hand side of \eqref{Analytic_Double_weight_1} can also be interpreted as a weighted sum. In order to derive the weights involved, we dissect the summand on the left. For a positive integer $n$, we have\begin{equation}\label{dissection}
\frac{q^{{(n^2+n)}/{2}}(-1;q)_n}{(q;q)^2_n} =  \frac{q^{{(n^2+n)}/{2}}}{(q;q)_n}\frac{2}{1-q^n} \frac{(-q;q)_{n-1}}{(q;q)_{n-1}}.
\end{equation} The first expression on the right \begin{equation}\label{Ferrer_distinct}\frac{q^{{(n^2+n)}/{2}}}{(q;q)_n}\end{equation} is the generating function for the number of partitions into exactly $n$ distinct parts \cite{Theory_of_Partitions}. We will think of these partitions to have the base color of white. The rational factor 
\begin{equation}\label{Ferrer_same_columns} \frac{2}{1-q^n}= 2+ 2q^n + 2q^{2n}+2q^{3n}+\dots
\end{equation}
 is the generating function for the number of partitions into parts each of size $n$ each time counted with weight 2, regardless of occurrence. We combine Ferrers diagrams of partitions enumerated by \eqref{Ferrer_distinct} and the conjugate of partitions counted by \eqref{Ferrer_same_columns} using column insertions. This yields the generating function for the number of partitions into exactly $n$ disctinct parts, where part $\lambda_n$ is counted with weight $2\lambda_n$. 

The column insertion is similar to the case in the 2-modular Ferrers diagrams as we exemplified in Table~\ref{Table_insertion example}. We embed a colored column from a conjugate of a colored partition counted by \eqref{Ferrer_same_columns} all the way left inside a Ferrers diagram counted by \eqref{Ferrer_distinct} without violating the definition of a partition. An example of column insertion is given in Table~\ref{column_insertion}.

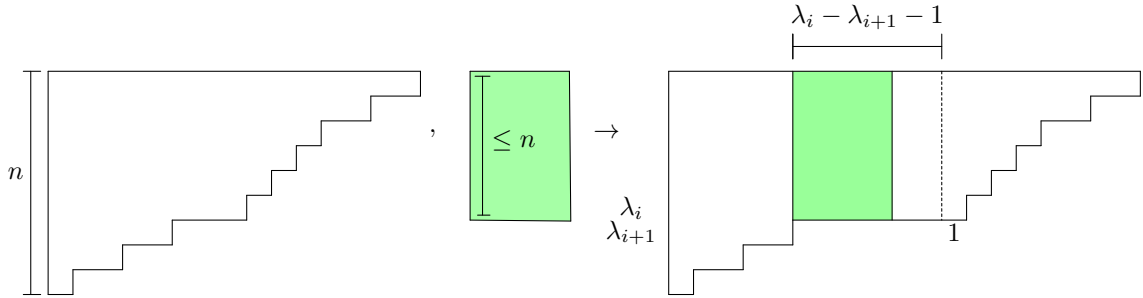
\begin{table}[htb]\caption{Column insertion}\label{column_insertion}
\definecolor{qqffqq}{rgb}{0.,1.,0.}
\definecolor{cqcqcq}{rgb}{0.7529411764705882,0.7529411764705882,0.7529411764705882}
\begin{tikzpicture}[line cap=round,line join=round,>=triangle 45,x=0.33cm,y=0.33cm]
\clip(-6.1,0.9) rectangle (40.1,13.1);
\fill[line width=0.pt,color=qqffqq,fill=qqffqq,fill opacity=0.35] (13.,10.) -- (17.,10.) -- (17.054927469276898,3.9609564812242373) -- (13.,4.) -- cycle;
\fill[line width=0.pt,color=qqffqq,fill=qqffqq,fill opacity=0.4] (26.,4.) -- (25.99278414969377,10.) -- (30.,10.) -- (30.019296207337426,4.) -- cycle;
\draw (-3.,2.)-- (-1.,2.);
\draw (-1.,3.)-- (-1.,2.);
\draw (-1.,3.)-- (1.,3.);
\draw (1.,3.)-- (1.,4.);
\draw (1.,4.)-- (4.,4.);
\draw (4.,4.)-- (4.,5.);
\draw (4.,5.)-- (5.,5.);
\draw (6.,7.)-- (7.,7.);
\draw (-3.,2.)-- (-3.,1.);
\draw (-3.,1.)-- (-4.,1.);
\draw (-4.,1.)-- (-4.,10.);
\draw (-4.,10.)-- (11.,10.);
\draw (11.,10.)-- (11.,9.);
\draw (11.,9.)-- (9.,9.);
\draw (9.,9.)-- (9.,8.);
\draw (9.,8.)-- (7.,8.);
\draw (7.,8.)-- (7.,7.);
\draw (13.,10.)-- (13.,4.);
\draw (13.,4.)-- (17.054927469276898,3.9609564812242373);
\draw (17.054927469276898,3.9609564812242373)-- (17.,10.);
\draw (17.,10.)-- (13.,10.);
\draw (21.,10.)-- (40.,10.);
\draw (36.,8.)-- (36.,7.);
\draw (36.,7.)-- (35.,7.);
\draw (35.,7.)-- (35.,6.);
\draw (35.,6.)-- (34.,6.);
\draw (34.,6.)-- (34.,5.);
\draw (34.,5.)-- (33.,5.);
\draw (33.,5.)-- (33.,4.);
\draw (26.,4.)-- (26.,3.);
\draw (26.,3.)-- (24.,3.);
\draw (24.,3.)-- (24.,2.);
\draw (24.,2.)-- (22.,2.);
\draw (22.,2.)-- (22.,1.);
\draw (22.,1.)-- (21.,1.);
\draw (21.,1.)-- (21.,10.);
\draw (19.5,4.5) node[anchor=center] {$ \lambda_i $};
\draw (19.5,3.5) node[anchor=center] {$ \lambda_{i+1} $};
\draw (29,12.2) node[anchor=center] {$ \lambda_i - \lambda_{i+1} -1$};
\draw (11.5,7.5) node[anchor=center] {,};
\draw (18.5,7.5) node[anchor=center] {$ \rightarrow $};
\draw (5.,5.)-- (5.,6.);
\draw (5.,6.)-- (6.,6.);
\draw (6.,6.)-- (6.,7.);
\draw (26.,4.)-- (33.,4.);
\draw (36.,8.)-- (38.,8.);
\draw (38.,8.)-- (38.,9.);
\draw (38.,9.)-- (40.,9.);
\draw (40.,9.)-- (40.,10.);
\draw (26,10.)-- (26.,4.);
\draw (30.,10.)-- (30,4.);
\draw [dash pattern=on 1pt off 1pt] (32,10.)-- (32.,4.);
\draw (32.5,3.5) node[anchor=center] {$1$};
\draw (26.,11.)-- (32.,11.);
\draw (26,10.5)-- (26,11.5);
\draw (32,10.5)-- (32,11.5);
\draw (26,10.5)-- (26,11.5);
\draw (-4.3,10)-- (-5.,10.);
\draw (-5.,1.)-- (-4.3,1);
\draw (-4.7,1)-- (-4.7,10);
\draw (-6,6.5) node[anchor=north west] {$n$};
\draw (13.3,4.2)-- (13.8,4.2);
\draw (13.2,9.8)-- (13.8,9.8);
\draw (13.5,9.8)-- (13.5,4.2);
\draw (13.5,8) node[anchor=north west] {$\leq n$};
\end{tikzpicture}
\end{table}

The expression \begin{equation*}
\frac{(-q;q)_{n-1}}{(q;q)_{n-1}}
\end{equation*} is the generating function for the number of partitions into parts $\leq n-1$, where every different sized part is counted with weight 2. After conjugating these partitions and inserting its columns to partitions into $n$ distinct parts, we see that there are $ 2(\lambda_i - \lambda_{i+1}-1)+1$ possible colorations between consecutive parts,  where at least one secondary color appears for $1\leq i\leq n-1$. To be more precise, there are $\lambda_i - \lambda_{i+1}-1$ columns coloring the space between $\lambda_{i+1}$ and $\lambda_{i}-1$ and each coloring comes with weight 2. This way we have the weight $2(\lambda_i - \lambda_{i+1}-1)+1$ where the extra 1 comes from the option of not having a colored column at all. Again these column insertions are demonstrated in Table~\ref{column_insertion}.

Hence, for a partition $\pi=(\lambda_1,\lambda_2,\dots)$, we have \begin{equation}\label{Double_weight_identity_1}
\sum_{n \geq 0} \frac{q^{(n^2+n)/{2}}(-1;q)_n}{(q;q)^2_n} = \sum_{\pi\in\mathcal{D}} \widetilde{\omega}_1 (\pi) q^{|\pi|},
\end{equation} where \begin{equation}\label{Double_weight_1}
\widetilde{\omega}_1 (\pi) := 2\lambda_{\nu(\pi)}\cdot\prod_{i=1}^{\nu(\pi)-1} (2\lambda_i-2\lambda_{i+1}-1).
\end{equation}
Similar to \eqref{weight_2}, we can change the product of $\widetilde{\omega}_1$ into a uniform product over the parts of a partition. With the custom choice that $\lambda_{\nu(\pi)+1} := -1/2$, we have \begin{equation}\label{Double_weight_1_1}
\widetilde{\omega}_1 (\pi) = \prod_{i=1}^{\nu(\pi)} (2\lambda_i-2\lambda_{i+1}-1).
\end{equation} Combining \eqref{Analytic_Double_weight_1}, \eqref{weights_of_overpartitions}, and \eqref{Double_weight_identity_1} yields
\begin{theorem}\label{AlladisDream1} \[\sum_{\pi\in\mathcal{D}} \widetilde{\omega}_1 (\pi) q^{|\pi|} = \sum_{\pi \in \mathcal{U}} {\omega_1'}(\pi)q^{|\pi|},\] where $\widetilde{\omega}_{1}(\pi)$ is as in \eqref{Double_weight_1_1} and ${\omega_1'}(\pi) = 2^{\nu_d(\pi)}$.
\end{theorem}

This is the first example of a weighted partition identity connecting $\mathcal{D}$ and $\mathcal{U}$ with strictly positive weights. The combinatorial interpretation of \eqref{Analytic_Double_weight_2} is going to provide a second example of a connection between $\mathcal{D}$ and $\mathcal{U}$ making use of a new partition statistic.

The left side of \eqref{Analytic_Double_weight_2} can be interpreted similar to \eqref{Analytic_Double_weight_1}. The weights associated with this case differ from the weight $\widetilde{\omega}_1$ only at the last part. For a partition $\pi=(\lambda_1,\lambda_2,\dots)$ with the custon definition that $\lambda_{\nu(\pi)+1}:=0$ we define the new weight uniformly as in \eqref{Double_weight_1_1}, \begin{equation}\label{the_big_weight_1}\widetilde{\omega}_2(\pi) =\prod_{i=1}^{\nu(\pi)} (2\lambda_i-2\lambda_{i+1}-1).\end{equation}
With this definition, we have the identity similar to \eqref{Double_weight_identity_1}, \begin{equation}\label{Double_weight_identity_2}\sum_{n \geq 0} \frac{q^{(n^2+n)/{2}}(-q;q)_n}{(q;q)^2_n} = \sum_{\pi\in\mathcal{D}} \widetilde{\omega}_2 (\pi) q^{|\pi|}.\end{equation}

In order to get the weights for the right side of \eqref{Analytic_Double_weight_2}, we modify that expression. We rewrite $(-q;q^2)_\infty$, the generating function for number of partitions into distinct odd parts, as
\begin{equation}\label{smallest_part}
(-q;q^2)_\infty = 1 + \sum_{n\geq 1} q^{2n-1} (-q^{2n+1};q^2)_\infty.
\end{equation} Note that the summands in \eqref{smallest_part} are generating functions for the number of partitions into distinct odd parts with the smallest part being equal to $2n-1$.

The right-hand side expression of the identity \eqref{Analytic_Double_weight_2} directly yields
\begin{align}
\label{initial_step_to_weights} \frac{(-q;q)_\infty}{(q;q)_\infty} \left(1+\sum_{n\geq 1} \frac{(-1)^n q^{2n-1}}{(-q;q^2)_{n}}\right) &= \frac{(-q^2;q^2)_\infty}{(q;q)_\infty} \left((-q;q^2)_\infty+\sum_{n\geq 1}(-1)^n q^{2n-1}(-q^{2n+1};q^2)_\infty\right).\\ 
\intertext{Employing \eqref{smallest_part}, combining sums and changing the summation indices $n\mapsto n+1$ on the right side of  \eqref{initial_step_to_weights}, we get}
\nonumber&\hspace{-2cm}=\frac{(-q^2;q^2)_\infty}{(q;q)_\infty}\left(1+ 2\sum_{n\geq 0} q^{4n+3} (-q^{4n+5};q^2)_\infty \right)\\
\label{Double_weight_core}&\hspace{-2cm}=\frac{(-q^2;q^2)_\infty}{(q^2;q^2)_\infty}\left(\frac{1}{(q;q^2)_\infty}+ \sum_{n\geq 0} \frac{1}{(q;q^2)_{2n+1}} \frac{2q^{4n+3}}{1-q^{4n+3}} \frac{(-q^{4j+5};q^2)_\infty}{(q^{4j+5};q^2)_\infty} \right).
\end{align}

We can interpret \eqref{Double_weight_core} as a combinatorial weighted identity over the set of unrestricted partitions, $\mathcal{U}$. Let $\pi=(\lambda_1,\lambda_2,\dots)$ be a partition. Let $\nu_{de}(\pi)$ be the number of different even parts. Let $\mu_{n,o}(\pi)$ denote the new partition statistic, defined as the number of different odd parts (without counting repetitions) $\geq n$ of $\pi$, for some integer $n$. Let 
\begin{equation}\label{truth_function}\chi(\textit{statement})=\left\{\begin{array}{cc}
1, &\text{if the \textit{statement} is true},\\
0, &\text{otherwise},
\end{array}\right. \end{equation}
be the \textit{truth} function. We define \begin{equation}\label{the_big_weight}
{\omega_2'}(\pi) = 2^{\nu_{de}(\pi)} \left( 1 + \sum_{i\geq 0} \chi((4i+3)\in\pi)2^{\mu_{4i+3,o}(\pi)} \right).
\end{equation} With these definitions and keeping \eqref{Double_weight_core} in mind, we have the weighted identity
\begin{equation}\label{the_big_weight_identity}
\frac{(-q;q)_\infty}{(q;q)_\infty} \left(1+\sum_{n\geq 1} \frac{(-1)^n q^{2n-1}}{(-q;q^2)_{n}}\right) = \sum_{\pi\in \mathcal{U}} {\omega_2'}(\pi) q^{|\pi|}.
\end{equation}

The emergence of this weight can be explained in two parts. The front factor of \eqref{Double_weight_core} (identity \eqref{weights_of_overpartitions} with $q\mapsto q^2$) yields the weight $2^{\nu_{de}(\pi)}$. This is easy to see as in the combined partition all of the parts coming from \eqref{weights_of_overpartitions} with $q\mapsto q^2$ can be thought of as even parts. The summation part of the weight \eqref{the_big_weight} comes from the respective summation in \eqref{Double_weight_core}\[\frac{1}{(q;q^2)_\infty}+ \sum_{n\geq 0} \frac{1}{(q;q^2)_{2n+1}} \frac{2q^{4n+3}}{1-q^{4n+3}} \frac{(-q^{4j+5};q^2)_\infty}{(q^{4j+5};q^2)_\infty}.\] The first term is the generating function for the number of partitions into odd parts where we count every partition once. The right summation is the weighted count of partitions into odd parts. For a non-negative integer $n$ the summand \[\frac{1}{(q;q^2)_{2n+1}} \frac{2q^{4n+3}}{1-q^{4n+3}} \frac{(-q^{4j+5};q^2)_\infty}{(q^{4j+5};q^2)_\infty}\] is the generating function for the number of partitions, where $4n+3$ appears as a part, every odd part less than $4n+3$ is counted once, and every different odd part $\geq 4n+3$ is counted with the weight 2. This yields the weight $2^{\mu_{4n+3,o}(\pi)}$ for a partition $\pi$.

Above observations \eqref{Double_weight_identity_2} and \eqref{the_big_weight_identity} combined with \eqref{Analytic_Double_weight_2} provide another new example of a relation between partitions into distinct parts and partitions into unrestricted parts with non-vanishing weights.

\begin{theorem}\label{AlladisDream2}\[\sum_{\pi\in\mathcal{D}} \widetilde{\omega}_2 (\pi) q^{|\pi|} = \sum_{\pi \in \mathcal{U}} {\omega_2'}(\pi)q^{|\pi|},\] where $\widetilde{\omega}_{2}(\pi)$ is as in \eqref{the_big_weight_1} and ${\omega_2'}(\pi)$ as in \eqref{the_big_weight}.
\end{theorem}

We would like to exemplify Theorem~\ref{AlladisDream2} in Table~\ref{Alladi_Big_Dream_Table}.

\begin{table}[htb]\caption{Example of Theorem~\ref{AlladisDream2} with $|\pi|=10$.}\label{Alladi_Big_Dream_Table}
\begin{center}\vspace{-.5cm}

\[\begin{array}{cc|cc|cc||cc}
\pi\in\mathcal{U}  & \omega_2' &  \pi\in\mathcal{U}& \omega_2' &  \pi\in\mathcal{U}& \omega_2'& \pi\in\mathcal{D} & \widetilde{\omega}_2\\[-2ex]& & & & & & &\\
(10) 		 & 2		&   (5,3,2) 		& 10  	&(3,3,3,1)		&   3& (10) & 19\\
(9,1)		 & 1		&   (5,3,1,1) 		&  5  	&(3,3,2,2)		&   6& (9,1) & 15\\
(8,2)		& 4			&   (5,2,2,1) 		&  2  	&	(3,3,2,1,1)		&   6& (8,2) & 33\\
(8,1,1) 	 & 2		&   (5,2,1,1,1)	 	&  2 	&	(3,3,1,1,1,1)	&   3	& (7,3) & 35\\
(7,3) 		 & 7		&   (5,1,1,1,1,1)  	& 1		&	(3,2,2,2,1)		&   6	& (6,4) & 21\\
(7,2,1) 	&  6		&   (4,4,2)  		& 	4	&	(3,2,2,1,1,1)	&   6	 & (6,3,1) & 15\\
(7,1,1,1) 	&  3		&   (4,4,1,1)  		&  2	&	(3,2,1,1,1,1,1)	&   6& (5,4,1) & 5\\
(6,4) 		 & 4		&    (4,3,3) 		&   6 	&(3,1,1,1,1,1,1,1)&  3	& (5,3,2) & 9\\
(6,3,1) 	 & 6		&   (4,3,2,1)		& 	12 	&   (2,2,2,2,2)		&   2	& (4,3,2,1) & 1\\
(6,2,2) 	&  4		&   (4,3,1,1,1)		&   6	&   (2,2,2,2,1,1)	&   2	& &\\
(6,2,1,1) 	&  4		&   (4,2,2,2)		&   4	&   (2,2,2,1,1,1,1)	&   2	& &\\
(6,1,1,1,1) & 	2		&   (4,2,2,1,1)		&   4	& 	(2,2,1,1,1,1,1,1)&   2	& &\\
(5,5) 		&  1		&   (4,2,1,1,1,1)	&   4  	&	(2,1,1,1,1,1,1,1,1)&  2	& &\\
(5,4,1) 	& 2			&   (4,1,1,1,1,1,1)	&   2	&   (1,1,1,1,1,1,1,1,1,1)&  1	& &\\  		
\end{array}\]
The summation of all $\omega_2'(\pi)$, or all $\widetilde{\omega}_{2}(\pi)$ for $|\pi|=10$ are the same and the sum equals 162.
\end{center}
\end{table}

In literature, there are many examples of partition identities with multiplicative weights. This is no different from the previous parts of this paper, such as Theorem~\ref{Alladi_weighted_sum},  \ref{Combinatorial_GG_identities_THM}, \ref{GG_1_combinatorial_theorem}, \ref{Half_Weight_Theorem_Most_General}. and \ref{AlladisDream1}. Theorem~\ref{AlladisDream2} is interesting not only because it gives a weighted connection between the sets $\mathcal{D}$ and $\mathcal{U}$, but also because of the appearance of the unusual additive weights.

The expression \eqref{Analytic_Double_weight_False_Theta}, which involves an order 3/2 false theta function, can be interpreted as a generating function for a weighted count of the ordinary partitions. The interpretation of the similar expression \eqref{Analytic_Half_weight_2}, which has an order 3 false theta function, required us to depart from the set of partitions with distinct odd parts $\P_{do}$ to an unexpected set $\mathcal{A}$ (with partitions not necessarily having distinct odd parts) with trivial weight 1 for each partition. Now we have a different situation. We stay with the set of all partitions $\mathcal{U}$, but the weights become non-trivial and, occasionally, zero. 


Recall that in frequency notation, a partition $\pi = (1^{f_1},2^{f_2},\dots)$, where $f_i(\pi)=f_i$ is the number of occurrences of $i$ in $\pi$. Let \begin{align}\label{last_weight}\omega_2^*(\pi) &= (1 - \chi(f_1(\pi)\geq 2))\prod_{n\geq 2} 2^{\chi(f_n(\pi)\geq 1)} +\\\nonumber&\hspace{-1cm}\sum_{j\geq 1}\left( \chi(f_{2j+1}(\pi)\leq 1) \chi(f_j(\pi)\geq 2) 2^{ \chi(f_j(\pi)\geq 3)}\prod_{i=1}^{j-1} \chi(f_i(\pi)\geq 3)  2^{\chi(f_i(\pi)\geq 4 )}\prod_{\substack{n>j,\\ n\not= 2j+1}} 2^{\chi(f_n(\pi)\geq 1)}\right), \end{align} where $\chi$ is defined in \eqref{truth_function}. We remark that the sum in $\omega_2^*(\pi)$ is finite as partitions are finite, and so $\chi(f_i(\pi)\geq 3)$ vanishes for any value of $i$ greater than the largest part of $\pi$. Then we have \begin{equation}
\frac{(-q;q)_\infty}{(q;q)_\infty}\sum_{j\geq 0} q^{{(3j^2+j)}/{2}}(1-q^{2j+1}) = \sum_{\pi\in\U} \omega_2^*(\pi) q^{|\pi|}.
\end{equation} 

This can be proven by doing cancellations with the front factor of the false theta function \eqref{False_theta2}:
\begin{equation}\label{False_theta2}\frac{(-q;q)_\infty}{(q;q)_\infty}\sum_{j\geq 0} q^{{(3j^2+j)}/{2}}(1-q^{2j+1}) = {(-q;q)_\infty}\sum_{j\geq 0} \frac{q^{{(3j^2+j)}/{2}}}{(q;q)_{2j}(q^{2j+2};q)_\infty}.
\end{equation}
The expression \eqref{False_theta2} is the generating function of partitions with weights $\omega_2^*$. The front factor $(-q;q)_\infty$ is the generating function for the number of partitions into distinct parts. Therefore, for our interpretation, every part can appear at least once. For a non-negative integer $j$ the summand is the generating function for the number of partitions, where $2j+1$ does not appear as a part, every number up to $j-1$ appears at least 3 times, and $j$ appears at least 2 times, as $(3j^2+j)/2=1+1+1+2+2+2+\dots +(j-1)+(j-1)+(j-1)+j+j$. 

This weight is also non-trivial and a sum of multiplicative terms. This is exemplified in Table~\ref{what_a_mess}.

\begin{table}[htb]\caption{Example of Theorem~\ref{AlladisDream2} with $|\pi|=10$.}\label{what_a_mess}
\begin{center}\vspace{-.5cm}
\[\begin{array}{cc|cc|cc}
\pi\in\mathcal{U}  & {\omega}_{2}^* &  \pi\in\mathcal{U} & {\omega}_{2}^* &  \pi\in\mathcal{U} & {\omega}_{2}^*\\[-2ex]& & & & &\\
(10) 		 & 2		&   (5,3,2) 		& 	8  	&	(3,3,3,1)				&   2	\\
(9,1)		 & 2		&   (5,3,1,1) 		&  	2  	&	(3,3,2,2)				&   4	\\
(8,2)		& 4			&   (5,2,2,1) 		&  	4  	&	(3,3,2,1,1)				&   0	\\
(8,1,1) 	 & 2		&   (5,2,1,1,1)	 	&  	8 	&	(3,3,1,1,1,1)			&   0	\\
(7,3) 		 & 4		&   (5,1,1,1,1,1)  	& 	4	&	(3,2,2,2,1)				&   4	\\
(7,2,1) 	&  4		&   (4,4,2)  		& 	4	&	(3,2,2,1,1,1)			&   6	\\
(7,1,1,1) 	&  4		&   (4,4,1,1)  		&  	2	&	(3,2,1,1,1,1,1)			&   4	\\
(6,4) 		 & 4		&   (4,3,3) 		&   4 	&	(3,1,1,1,1,1,1,1)		&  	2	\\
(6,3,1) 	 & 4		&   (4,3,2,1)		& 	8 	&   (2,2,2,2,2)				&   2	\\
(6,2,2) 	&  4		&   (4,3,1,1,1)		&   4	&   (2,2,2,2,1,1)			&   2	\\
(6,2,1,1) 	&  4		&   (4,2,2,2)		&   4	&   (2,2,2,1,1,1,1)			&   8	\\
(6,1,1,1,1) & 	4		&   (4,2,2,1,1)		&   4	& 	(2,2,1,1,1,1,1,1)		&   6	\\
(5,5) 		&  2		&   (4,2,1,1,1,1)	&   8  	&	(2,1,1,1,1,1,1,1,1)		&  	4	\\
(5,4,1) 	& 4			&   (4,1,1,1,1,1,1)	&   4	&   (1,1,1,1,1,1,1,1,1,1)	&  	2	\\  		
\end{array}\]
The summation of all $\omega^*_{2}(\pi)$ values for $|\pi|=10$ equals 162, as in the values of Table~\ref{Alladi_Big_Dream_Table}.
\end{center}
\end{table}

Hence, we get the similar result to Theorem~\ref{Half_Weight_Theorem_Most_General}:

\begin{theorem} \[\sum_{\pi\in\mathcal{D}} \widetilde{\omega}_2 (\pi) q^{|\pi|} = \sum_{\pi \in \mathcal{U}} {\omega_2'}(\pi)q^{|\pi|}= \sum_{\pi \in \mathcal{U}} {\omega_2^*}(\pi)q^{|\pi|},\] where weights $\widetilde{\omega}_{2}(\pi)$, ${\omega_2'}(\pi)$, and $\omega_2^*$ are as in \eqref{the_big_weight_1}, \eqref{the_big_weight} and \eqref{last_weight}, respectively.
\end{theorem}
\section{A Weighted Partition Identity Related to $\frac{1}{(q;q)_\infty}\sum_{j=0}^\infty q^{(3j^2+j)/2}(1-q^{2j+1})$}
\label{Section_Restricted}

In Section~\ref{GollnitzWeights}, we have proven Theorems~\ref{GG_1_combinatorial_theorem} and \ref{Half_Weight_Theorem_2} involving partitions with distinct odd parts counted with trivial weights. In this section we will derive another partition identity involving partitions with distinct odd parts, this time with non-trivial weights. To this end we prove the following theorem.

\begin{theorem}\label{RestrictedTHM}\begin{equation}\label{RestrictedEQN}\frac{1}{(q^2;q^2)_\infty}\sum_{n=0}^{\infty} q^{(2n+1)n}(-q^{2n+2};q)_\infty = \frac{1}{(q;q)_\infty}\sum_{j=0}^\infty q^{(3j^2+j)/2}(1-q^{2j+1}). \end{equation}
\end{theorem}

\begin{proof} This theorem amounts to manipulating the equality of \eqref{Fine2} and \eqref{Fine3}. We point out that doing the even--odd index split of the summand of \eqref{Fine3} and using \[\frac{q^{(2n+1)n}}{(-q;q)_{2n}} - \frac{q^{(2n+1)n+(2n+1)}}{(-q;q)_{2n+1}} = \frac{q^{(2n+1)n}}{(-q;q)_{2n+1}}\] yields
\begin{equation}\label{Example10_with_x_substituted} 
\sum_{n=0}^{\infty} \frac{ q^{(2n+1)n}}{(-q;q)_{2n+1}} = \sum_{j=0}^\infty q^{(3j^2+j)/2}(1-q^{2j+1}).
\end{equation}
The identity \eqref{Example10_with_x_substituted} appears in the Ramanujan's lost notebooks \cite[(9.4.4), p. 233]{LostNotebook_1}. Multiplying both sides of \eqref{Example10_with_x_substituted} with \[ \frac{(-q;q)_\infty}{(q^2;q^2)_\infty} = \frac{1}{(q;q)_\infty},\] and doing the necessary simplifications on the left, we arrive at \eqref{RestrictedEQN}.
\end{proof}

Next we define two sets of partitions. Let $\P_{dom}$ be the set of partitions with distinct odd parts, where the smallest positive integer that is not a part is odd, and let $\U_{ic}$ be the set of ordinary partitions subject to the initial condition that if $2j+1$ is the smallest positive odd number that is not a part of the partition, then every even natural number $\leq j$ appears as a part, and all the odd natural numbers $\leq j$ appear at least twice in this partition. We rewrite \eqref{RestrictedEQN} suggestively as
\begin{equation}\label{R_rewritten}\sum_{n=0}^{\infty} \frac{q^{(2n+1)n}}{(q^2;q^2)_{n}}\frac{(-q^{2n+2};q^2)_\infty}{(q^{2n+2};q^2)_\infty} (-q^{2n+3};q^2)_\infty = \sum_{j=0}^\infty \frac{q^{(3j^2+j)/2}}{(q;q)_{2j}(q^{2j+2};q)_\infty}\end{equation} to show that the left and the right sides of \eqref{RestrictedEQN} are related with counts for the partitions from the sets $\P_{dom}$ and $U_{ic}$, respectively. Observe that \[(2n+1)n = 1+2+\dots+2n\] and \[\frac{q^{(2n+1)n}}{(q^2;q^2)_{n}}\] is the generating function for the number of partitions with distinct odd parts where every part is $\leq 2n$ and every integer $\leq 2n$ appears at least once. The factor \[\frac{(-q^{2n+2};q^2)_\infty}{(q^{2n+2};q^2)_\infty}\] is the generating function for the number of partitions into even parts $\geq 2n+2$ where each different even part is counted with weight 2. Putting the factors in the left-hand summand of \eqref{R_rewritten} together, we see that the left side sum is a weighted count of partitions from $\P_{dom}$.  Also note that \[(3j^2+j)/2 = (1+2+3+\dots+j)+(1+3+5+\dots+(2j-1)),\] which is enough to see that the right side of \eqref{R_rewritten} is the generating function for the number of partitions from $\U_{ic}$. These observations prove the following

\begin{theorem}\label{Restricted_Comp_THM} \begin{equation*}
\sum_{\pi\in\P_{dom}} 2^{\tau(\pi)} q^{|\pi|} = \sum_{\pi\in\U_{ic}} q^{|\pi|},
\end{equation*}
where, for a partition $\pi$, $\tau(\pi)$ is the number of different even parts of $\pi$ larger than the smallest positive odd integer that is not a part of $\pi$.
\end{theorem}

We conclude with an example of this result in Table~\ref{Table_Restricted}.

\begin{table}[htb]\caption{Example of Theorem~\ref{Restricted_Comp_THM} with $|\pi|=8$.}\label{Table_Restricted}
\begin{center}\vspace{-.5cm}
\[\begin{array}{cc||c}
\pi\in\P_{dom}  &  2^{\tau(\pi)} &  \pi\in\U_{ic} \\[-2ex]& &  \\
(8) 		& 2		&   (8) 		 		\\
(6,2)		& 4		&   (6,2) 		  		\\
(5,3)		& 1		&   (6,1,1) 		 	\\
(5,2,1) 	& 1		&   (5,3)	 	 		\\
(4,4) 		& 2		&   (5,1,1,1)  			\\
(4,2,2) 	& 4		&   (4,4)  			\\
(2,2,2,2) 	& 2		&   (4,2,2)  			\\
 			& 		&   (4,2,1,1) 		\\
 		 	& 		&   (4,1,1,1,1)		\\
 			&  		&   (3,3,2)		\\
 			&  		&   (3,2,1,1,1)		\\
 			&  		&   (2,2,2,2)		\\	
 			&  		&   (2,2,2,1,1)		\\	
 			&		&	(2,2,1,1,1,1)	\\
 			&		&	(2,1,1,1,1,1,1)	\\
 			&		&	(1,1,1,1,1,1,1,1)\\
\end{array}\]
The sum of the weights $2+4+1+1+2+4+2=16$ is the same as the number of partitions from $\U_{ic}$ with $|\pi|=8$.
\end{center}
\end{table}

\section{Acknowledgement}
The authors would like to thank George E. Andrews, Krishna Alladi and Andrew V. Sills for their kind interest and helpful suggestions. The authors would also like to thank Jeramiah A. Hocutt and Benjamin P. Russo for their careful reading of the manuscript.


\end{document}